\documentclass{article}
\usepackage[american]{babel}
\usepackage[utf8]{inputenc}

\usepackage{amsmath, amssymb}   
\usepackage{amsthm}
\usepackage{mathrsfs}
\usepackage{mathtools}
\usepackage{stix}
\usepackage{color}
\usepackage{changes}

\usepackage{makeidx}
\usepackage{acronym}
\usepackage{algorithm}
\usepackage{algpseudocode}
\usepackage{longtable} 

\usepackage{enumitem}
\usepackage{adjustbox} 
\usepackage{tikz}
\usetikzlibrary{calc, shapes, positioning, decorations}

\usepackage{pgfplots}
\usepackage{pgfplotstable}
\pgfplotsset{compat=1.9}

\usepackage[T1]{fontenc}
\usepackage{titlesec}

\theoremstyle{plain}
\newtheorem{theorem}{Theorem}[section]

\newtheorem{conjecture}[theorem]{Conjecture}

\theoremstyle{definition}
\newtheorem{definition}[theorem]{Definition}

\newtheorem{question}[theorem]{Question}

\newcommand{\C}{\mathbb{C}}
\newcommand{\Z}{\mathbb{Z}}

\newcommand{\diff}[2]{#1 \cdot #2^{-1}}

\renewcommand{\paragraph}[1]{\leavevmode\vspace{0,2cm}\newline{\it #1}\\

}

\let\oldReturn\Return
\renewcommand{\Return}{\State\oldReturn}

\title{Results on formally dual sets in finite abelian groups of size 64 obtained from a graph search algorithm}
\author{Robert Schüler}
\date{February 2023}

\begin{document}

\maketitle

\begin{abstract}
    We shortly present two small results regarding the study formal duality in finite abelian groups as introduced by Cohn, Kumar, Reiher and
    Schürmann.
    In particular, we give a new example of a formally self dual set in $\Z_2^2\times\Z_4^2$ and computed nonexistence of primitive formally dual sets of size $8$ in $\Z_8^2$.
\end{abstract}

\section{Introduction}

Formally duality in finite abelian groups is a concept introduced by Cohn, Kumar and Schürmann while studying energy minimization problems (see \cite{cohn2009ground}, \cite{cohn2014formal}) and is defined as follows:

\begin{definition}\label{def:formal-duality}
Let $G$ be some (multiplicative) finite abelian group and
$\hat G$ be its dual group, i.e., the group of homomorphisms from $G$ to $\C^\ast$. Two sets $S \subset G$ and $T \subset \hat G$
form  a formally dual pair, if for all
$\chi\in\hat G$ we have

\begin{equation}
\left|\chi(S)\right|^2 = \frac{|S|^2}{|T|}\nu_T(\chi),\label{eq:def_fd}
\end{equation}
or equivalently if for all $g\in G$ we have
\begin{equation}
\left|g(T)\right|^2 = \frac{|T|^2}{|S|}\nu_S(g),\label{eq:equivalent_def_fd}
\end{equation}
where 
$$\nu_T(\chi) = |\{(\phi,\psi)\in T\times T \ : \ \diff \phi\psi = \chi\}|$$ is called the \emph{weight enumerator} of $T$ and we use the notation $\chi(S) = \sum_{x\in S} \chi(x)$ and $g(T) = \sum_{\chi\in T} \chi(g)$.

\bigskip

A set $S$ is called a \emph{formally dual set} if there is a set $T$ such that $S$ and $T$ form a formally dual pair.
\end{definition}

A finite abelian group is always isomorphic to its dual group. Thus, by choosing
an isomorphism $\Delta:G\rightarrow \hat G$ we can suppose without loss of generality that $T\subset G$.
Equivalently, we can choose a pairing, that is a nondegenerate bilinear form $\left< . , .\right> : G \times G \rightarrow C*$.
We call a set $S\subset G$ \emph{formally self dual} if there is an isomorphism $\Delta:G\rightarrow \hat G$ such that $S$ and $\Delta(S)$ are a formally dual pair (see also \cite{Koelsch2021}).

\bigskip

Formal duality in cyclic groups has been studied in \cite{schuler2017Cyclic}, \cite{xia2016classification}, \cite{malikiosis2017cyclic}.
Discussions of formal duality in general abelian groups and their relation to relative difference sets can be found in \cite{li2018abelian}, \cite{schuler2019phd}.
In \cite{li2019constructions} and \cite{li2020direct_construction} the authors discussed constructions of formal dual sets.
Lastly, in \cite{Koelsch2021} you can find a discussion of formal self dual sets and the connection to formally dual codes.

A regularly discussed research question concerns primitive formally dual sets, where a set $S\subset G$ is said to be \emph{primive} if \emph{none} of the following holds:
\begin{enumerate}
    \item $S\subset v\cdot H$ for some $H < G$,
    \item $S$ is a union of cosets with respect to a non-trivial subgroup of $G$.
\end{enumerate}

It has been shown that any pair of formally dual sets can be reduced to a pair of primitive formally dual sets.

In the literature, quite some effort has been taken to answer the following question:

\begin{question}
Which finite abelian groups contain primitive formally dual sets?
\end{question}

This question is still wildly open. However, it has been solved (using computers for some cases) for groups of order $< 64$. You can find an overview of these smaller groups in \cite[Appendix A]{li2018abelian}.

Some cases of groups of order $64$ have also been solved as can be seen in \cite{Koelsch2021}[Appendix, Table 2], leaving a total of $11$ open cases for the tuple $(G, |S|)$.

\section{New results in groups of order $64$}

In \cite{schuler2019phd}[Chapter 7] described a way to organize subsets of a given group $G$ in a directed tree.
We recall the main advantages and disadvantages of this approach:
\begin{enumerate}
    \item The tree structure makes it easy to perform graph search algorithms.
    \item Some branches of the tree are very similar to each other in the sense, that one branch contains primitive formally dual subsets if and only if the other contains primitive formally dual subsets. In these cases it is sufficient to only search one branch, reducing the number of sets we need to examine.
    \item If there exists a primitive formally dual subset of given size in $G$, the algorithm will find it. The existence can then be proved by the computed example.
    \item If there exists no primitive formally dual subset of the given size in $G$, the algorithm will terminate which certifies the non-existence. However, it does not provide a certificate to easily check the non-existence.
\end{enumerate}
We used the computer search method proposed in \cite{schuler2019phd} exploiting translations and automorphisms in order to produce the following new example in roughly $600$ days:

\begin{theorem}
If $G = \Z_2^2\times\Z_4^2$ then there exists a primitive formally self dual subset $S\subset G$ with $|S| = 8$.
\end{theorem}
\begin{proof}
    Take for example
    $$S = \{(0,0,0,0), (0,0,0,1), (0,0,0,2), (0, 0, 1,0), (0,0,2,1), (0,1,0,0), (1,0,0,0), (1,1,3,2)\}$$
    and the pairing
    $$\left<(a, b, c, d), (a', b', c', d')\right> = e^{(aa' + bb')i\pi + (cd' + c'd)i\pi/2}.$$
    We easily can compute
    $$8\cdot\nu_S(g) = |g(S)|^2$$
    for all $g\in G$ yielding that $S$ is a formally self dual set under the isomorphism equivalent to the given pairing.
    Furthermore, $S$ generates $G$ and therefore is primitive by \cite{cohn2014formal}[Lemma 4.2].
\end{proof}

We used the same searching techniques on the group $\Z_8^2$ where the graph search algorithm terminated without any result. On the one hand, this is a proof of the non-existence. On the other hand, we will not rely on the reader to believe that we did the computation. Therefore we formulate the result as conjecture, hoping that a more sophisticated proof can be found:

\begin{conjecture}
    There is no primitive formally dual subset of size $8$ in $\Z_8^2$.
\end{conjecture}

\bibliographystyle{alpha}
\bibliography{references}

\begin{thebibliography}{CKRS14}

\bibitem[CKRS14]{cohn2014formal}
H.~Cohn, A.~Kumar, C.~Reiher, and A.~Sch\"{u}rmann.
\newblock Formal duality and generalizations of the {P}oisson summation
  formula.
\newblock 625:123--140, 2014.

\bibitem[CKS09]{cohn2009ground}
H.~Cohn, A.~Kumar, and A.~Sch\"urmann.
\newblock Ground states and formal duality relations in the gaussian core
  model.
\newblock {\em Phys. Rev. E}, 80:061116, Dec 2009.

\bibitem[KS21]{Koelsch2021}
Lukas K\"{o}lsch and Robert Sch\"{u}ler.
\newblock Formal self duality.
\newblock {\em Cryptogr. Commun.}, 13(5):815--836, 2021.

\bibitem[LP19]{li2019constructions}
S.~Li and A.~Pott.
\newblock Constructions of primitive formally dual pairs having subsets with
  unequal sizes.
\newblock {\em J. Combin. Des.}, 27(12):703--733, 2019.

\bibitem[LP20]{li2020direct_construction}
S.~Li and A.~Pott.
\newblock A direct construction of primitive formally dual pairs having subsets
  with unequal sizes.
\newblock {\em Cryptogr. Commun.}, 12(3):469--483, 2020.

\bibitem[LPS19]{li2018abelian}
S.~Li, A.~Pott, and R.~Sch\"{u}ler.
\newblock Formal duality in finite abelian groups.
\newblock {\em J. Combin. Theory Ser. A}, 162:354--405, 2019.

\bibitem[Mal18]{malikiosis2017cyclic}
R.D. Malikiosis.
\newblock Formal duality in finite cyclic groups.
\newblock {\em Constructive Approximation}, Mar 2018.

\bibitem[Sch17]{schuler2017Cyclic}
R.~Sch\"{u}ler.
\newblock Formally dual subsets of cyclic groups of prime power order.
\newblock {\em Beitr. Algebra Geom.}, 58(3):535--548, 2017.

\bibitem[Sch19]{schuler2019phd}
R.~Sch\"{u}ler.
\newblock Thesis: Formal duality.
\newblock http://purl.uni-rostock.de/rosdok/id00002581, 2019.

\bibitem[Xia16]{xia2016classification}
J.~Xia.
\newblock Classification of formal duality with an example in sphere packing.
\newblock
  https://math.mit.edu/research/undergraduate/urop-plus/documents/2016/Xia.pdf,
  2016.
\newblock Accessed 2019-01-31.

\end{thebibliography}

\end{document}